\documentclass[11pt]{article}

\usepackage{amssymb}
\usepackage{latexsym}
\newtheorem{theorem}{Theorem}%[section]
\newtheorem{lemma}[theorem]{Lemma}

\newcommand{\proof}{\noindent {\it Proof.\hspace{4mm}}}
\newcommand{\qed}{\hfill \mbox{$\Box$}}
\newtheorem{theirtheorem}{Theorem}

\newtheorem{theirconj}{Conjecture}
\newcommand{\N}{{\mathbb N}}

\newcommand{\subgp}[1]{\langle{#1}\rangle}
\newcommand{\beeq}{\begin{eqnarray*}}
\newcommand{\eneq}{\end{eqnarray*}}

\begin{document}
\title{A note on  Pollard's Theorem }
\author{ {Y. O. Hamidoune}\thanks{
UPMC Univ Paris 06,
 E. Combinatoire, Case 189, 4 Place Jussieu,
75005 Paris, France.}
\and {O. Serra}\thanks{
 Universitat
Polit\`ecnica de Catalunya,
Jordi Girona, 1,
E-08034 Barcelona, Spain.}
}

\date{}

\maketitle

\begin{abstract}
Let $A,B$ be nonempty subsets of a an abelian group $G$. Let
$N_i(A,B)$ denote the set of elements of $G$ having $i$ distinct
decompositions as a product of an element of $A$ and an element of
$B$. We prove that
$$
\sum _{1\le i \le t} |N_i (A,B)|\ge  t(|A|+|B|- t-\alpha+1+w)-w,
$$
where $\alpha $ is the largest size of a  coset contained in $AB$
and $w=\min (\alpha-1,1)$, with a strict inequality if  $\alpha\ge
3$ and  $t\ge 2$, or if  $\alpha\ge 2$ and $t= 2$.

This result is a local extension of results by Pollard and
Green--Ruzsa and extends also  for $t>2$ a recent result of
Grynkiewicz,  conjectured by Dicks--Ivanov (for non necessarily
abelian groups) in connection to the famous Hanna Neumann problem in
Group Theory.
\end{abstract}

\hspace{1cm}{\bf MSC Classification:} 11B60, 11B34, 20D60.

\section{Introduction}

Groups will be written multiplicatively.

Let $G$ be a group and let $A, B$ be two nonempty subsets
of $G$. The {\em representation function} $r_{A,B}:G\rightarrow \N$
is defined by the relation $$r_{A,B}(x)=|(xB^{-1})\cap A|.$$

The subgroup generated by a subset $S$ will be denoted by
$\subgp{S}$.

Let  $X$ be a subset  of an  abelian group $G$.
Recall that the subgroup $\{x\in G :Xx=X\}$ is called the {\em period} of $X$.

Our notations follow almost everywhere the terminology of Nathanson
\cite{natlivre}. In particular  we write
$$N_t (A,B)=\{x\in AB : r_{A, B}(x)\ge t\}.$$

Pollard \cite{P1}   proved  the following remarkable result:

\begin{theirtheorem}[Pollard \cite{P1}]\label{pol1}
Let $A$ and $B$ be nonempty subsets of  a group with a prime order
$p$ and let $t\le \min (|A|,|B|).$ Then

$$\sum _{1\le i \le t} |N_i (A,B)| \ge t\min (p,|A|+|B|-t).$$
\end{theirtheorem}
Very recently Nazarewicz, O'Brien,  O'Neill and C. Staples
\cite{naza} obtained     the equality cases in Pollard's Theorem
\ref{pol1}. A restricted--sum version of the inequalities  is
obtained by Caldeira and Dias da Silva \cite{dias}.

Under a Chowla's type condition, Pollard \cite{P2} proved a
generalization of his theorem to composite moduli. He also suggested
the problem of finding  conditions implying the validity of his
inequalities in an abelian group.

In connection with old questions on sum-free sets \cite{GR}, the
following generalization of Pollard's Theorem to arbitrary abelian
groups was given by Green--Ruzsa:

\begin{theirtheorem}[Green--Ruzsa \cite{GR}]
Let $G$ be an abelian group and let $A$ and $B$ be nonempty subsets
of $G$ and let $t\le \min (|A|,|B|).$ Then $$\sum _{1\le i \le t}
|N_i (A,B)|\ge t\min (|G|,|A|+|B|-\mu (G)-t+1),$$ where $\mu (G)$
denotes the cardinality of the largest   subgroup of $G$ different
from $G$.\label{GR1}
\end{theirtheorem}

In the prime case $\mu (G) =1$ and hence Green--Ruzsa Theorem
reduces to Pollard's Theorem \ref{pol1}.

Kemperman \cite{kempcompl} proved a non-abelian counter-part of
Kneser's Theorem  (see Theorem \ref{kneser} in Section 2)  stating
that there is a finite subgroup $H$ such that $ |AB| \ge
|A|+|B|-|H|$, and $aHb\subset AB$ for some $a\in A$ and $b\in B$.
Dicks--Ivanov \cite{Dicks} obtained a generalization of this result
in connection with a famous problem in Group Theory and formulated
the following:

\begin{theirconj}[Dicks--Ivanov \cite{Dicks2,Dicks3}]\label{Dicks}
Let $A,B$ be finite nonempty subsets of a group $G$ with $2\le
|B|\le |A|.$
Then one of the following conditions holds:
\begin{itemize}
 \item [(i)] $ |N_1 (A,B)|+|N_2 (A,B)| \ge 2(|A|+|B|-2)$,
 \item [(ii)] $N_2 (A,B)$ contains a  left coset with cardinality $\ge 3$.
 \end{itemize}
\end{theirconj}

Grynkiewicz \cite[Theorem 3.1]{david} proved  that
 $ |N_1 (A,B)|+|N_2 (A,B)| \ge 2(|A|+|B|-\max (2,|H|))$, where $H$ is the period of $N_2(A,B)$.
Clearly this result implies the validity of the Dicks--Ivanov
Conjecture in the abelian case.
 Two  easier proofs of this conjecture in the abelian case
 were proposed later  by  Dicks--Ivanov \cite{Dicks3} and  by the authors.
 Also Grynkiewicz investigates in \cite{david}
 conditions implying the validity of the inequality
 $\sum _{1\le i \le t} |N_i (A,B)|\ge t(|A|+|B|- 2t+1)$.

Let $\alpha (A,B)$ denote the cardinality of the largest left coset
(non necessarily proper) contained in $AB$. In the spirit of the
Dicks--Ivanov Conjecture, we prove  the next local Green--Ruzsa type
Theorem:

\begin{theorem}\label{main}
Let $t$ be an integer and $A,B$ be finite  subsets of
an abelian group $G$ with $1\in  A\cap B$ and
 $|A|\ge |B|\ge t\ge 1.$ Set  $\alpha =\alpha(A,B)$ and $w=\min (\alpha-1,1)$.
Then \begin{equation}\sum _{1\le i \le t} |N_i (A,B)|\ge  t(|A|+|B|-
t-\alpha+1+w)-w.\label{MAIN}\end{equation} Moreover the inequality (\ref{MAIN}) is strict  in the
following cases:
\begin{itemize}
  \item [{\rm (I)}] $\alpha\ge 3$ and $t\ge 2$,
  \item [{\rm (II)}]  $\alpha\ge 2$ and $t= 2$.
\end{itemize}

\end{theorem}

 For $\alpha=1$ or $t=1$,  the inequality (\ref{MAIN}) is not necessarily
strict. By Theorem \ref{main}, the inequality (\ref{MAIN}) is strict
for $t>1$, unless $t\ge 3$ and $\alpha =2$. We suspect that it is
still strict in this case as well.

For $t=2$ Theorem \ref{main}, follows  from the result
of Grynkiewicz mentioned above.

\section{Preliminaries}
We shall need the following results:

\begin{lemma}\label{ph} (folklore)
Let $G$ be a finite  abelian group. Let $A,B$ be nonempty subsets such that
$|A|+|B|\ge |G|+t$. Then for every $x\in G$, $r_{A,B}(x)\ge t.$ In
particular $\sum _{x\in G} r_{A, B}(x)\ge t|G|$.
\end{lemma}

\proof We have clearly $|(xB^{-1})\cap A|\ge t.$ \qed

\begin{theirtheorem}[Kneser \cite{KN3, natlivre}]\label{kneser}
Let $G$ be an abelian group and let $A, B\subset G$ be finite
subsets of $G$ with $|AB|\le |A|+|B|-2$. Let $H$ be the period of $AB$.
Then  $|H|\ge 2$. Moreover
\begin{equation}\label{KN1}|AB|=|AH|+|BH|-|H|,
\end{equation}
\begin{equation}\label{KN2}N_2(A,B)=AB.
\end{equation}
\end{theirtheorem}

The equality (\ref{KN2}) is known as the Kemperman-Scherk Theorem.
It is proved by Scherck \cite{scherck} in the abelian case  and by
Kemperman \cite{kempcompl} in the non-abelian one. But in the
abelian case it can be easily derived from (\ref{KN1}) as follows:

Take $x\in AB$. Then
$x\in A_1B_1$, where $A_1$ (resp. $B_1$ ) is the trace of some coset on $A$ (resp; $B$). Now
$2|H|-|A_1|-|B_1|\le |AH|+|BH|-|A|-|B|=|AB|+|H|-|A|-|B|\le |H|-2.$
It follows that $|xA^{-1}\cap B|\ge 2$.\qed

The following lemma is implicitly proved by Pollard in  \cite{P1}. We give few hints
for its proof in order to make the present work self-contained:

\begin{lemma}\label{basic}
Let $G$ be an abelian group. Let $A,B$ be nonempty sets with $|B|\le
|A|$. Let $t,v$ be integers such that $0\le v\le t$ and let $z\in
G$.

\begin{equation}\label{POLD1}
 |N_t(A,B)|=|N_t(zA,B)|,
 \end{equation}
 \begin{equation}
 \sum _{1\le i \le |B|}|N_i(A,B)|=|A||B|,\label{POLD2}
 \end{equation}
\begin{equation}
\sum _{1\le i \le t} |N_i (A,B)|\ge \sum _{1\le i \le v} |N_i (A\cap B,A\cup B)|+
\sum _{1\le i \le t-v} |N_i (A\setminus B,B\setminus A)|.\label{POLD3}
\end{equation}

\end{lemma}

\begin{proof}
The proof of (i) is obvious.
We have
 $$\sum _{1\le i \le |B|}|N_i(A,B)|=\sum_{x\in G} r_{A, B}(x)=\sum_{x\in A} |xB|=|A||B|,$$ and hence (ii) holds.
 We have also
\begin{eqnarray*}
\sum _{1\le i \le t}|N_i(A,B)|&=&\sum _{ x\in G}\min (t,r_{A,B}(x))\\&=& \sum _{ x\in G}\min (t,
r_{A\setminus B,B\setminus A}(x)+r_{A\cap B,B\setminus A}(x)+r_{A\cap B,B\cap A}(x)
+r_{A\setminus B,A\cap B}(x))\\
&=& \min (t,r_{A\setminus B,B\setminus A}(x)+r_{A\cap B,A\cup B}(x))\\
&\ge&  \sum _{1\le i \le v} |N_i (A\cap B,A\cup B)|+
\sum _{1\le i \le t-v} |N_i (A\setminus B,B\setminus A)|.
\end{eqnarray*}\qed\end{proof}

 We need the following lemma:

\begin{lemma}\label{reduct}
Let  $ A,B$  be  finite nonempty subsets of an abelian group $G$ and let $a\in G$. Then
 \begin{description}
 \item[{\rm (i)}] $N_t(aA, B)= aN_t(A,B),$
\item[{\rm (ii)}] $N_t((A\cap B),(A\cup B))\subset N_t(A,B)$.
\item[{\rm (iii)}] $\alpha (aA\cup B,aA\cap
B)\le \alpha (A,B).$
\end{description}
\end{lemma}
The proof is easy.

\section{ Proof of Theorem \ref{main}}

The proof is by induction on $|B|.$

If $|B|=t$ the result follows from (\ref{POLD2}) and the inequality
is strict.

Assume first  that $AB=A$. Then $A$ has a decomposition
$A=A_1\cup\cdots\cup A_k$ as the union of $\subgp{B}$-cosets,  and
$\alpha \ge |\subgp{ B}|\ge |B|$.

 By Lemma \ref{ph} we   have
$\sum_{1\le i \le t} N_i(A_j,B)=t|A_j|$ and
$$
 \sum _{1\le i \le t} |N_i (A,B)|=\sum _{1\le i \le t, 1\le j\le k} |N_i
 (A_j,B)|=t|A|,$$

which proves the result and the inequality (\ref{MAIN}) is strict if $t>1$. So we
may assume $AB\neq A$.

Then there is  $b\in B$ and $a\in A$ such that $b\not\in Aa^{-1}$.
In particular $1\le |B\cap (Aa^{-1})|<|B|$. Put $A'=Aa^{-1}$ and
$v=|A'\cap B|$. Set $S=A'\setminus B$ and
$T=B\setminus A'$.

Using Lemma \ref{basic} and Lemma \ref{reduct}, for every $u$ with
$0\le u \le t$, we have
    \begin{eqnarray}
    \sum _{1\le i \le t} |N_i (A,B)|&\ge &
    \sum _{1\le i \le u} |N_i (A'\cap B,A'\cup B)|+ \sum _{1\le i \le t-u} |N_i (S, T)|.
    \label{EQAA}
    \end{eqnarray}
Suppose first that $t\le v$. Then, by  (\ref{EQAA}) applied with
$u=t$  and the induction hypothesis, we have
\begin{eqnarray*}
\sum_{1\le i \le t} |N_i (A,B)| &\ge&t(
|A|+|B|-t-\alpha+1+w)-w,\end{eqnarray*}
 Also the inequality (\ref{MAIN})is
strict in case (II) by induction and in case (I)   either by
induction or if $2=\alpha (A'\cup B, A'\cap B)<\alpha$.

% We will not insist on the case where $\alpha (A'\setminus B,B\setminus A')<\alpha$, since it is a much easier.

Suppose now that $t>v$.
Since  $v=|B\cap A'|$, we have by (\ref{POLD2})
\begin{equation}\label{eq:1}
\sum _{1\le i \le v} |N_i (A'\cap B,A'\cup B)|= v|A'\cup B|=
v(|A|+|B|-v).
\end{equation}

We now estimate the second summand in the left hand side of
(\ref{EQAA}) with $u=v$.

 By the induction hypothesis we have
\begin{equation}\label{eq:2}
\sum _{1\le i \le t-v} |N_i (S,T)|\ge (t-v)
(|A|+|B|-2v-(t-v)-\alpha+1+w)-w.
\end{equation}

By adding the second terms in  (\ref{eq:1}) and (\ref{eq:2})  we
have,
\begin{eqnarray*}
\sum _{1\le i \le t} |N_i (A,B)|&\ge& v(|A|+|B|-v)+(t-v)(|A|+|B|-t-v-\alpha+1+w)-w\\
&=& t(|A|+|B|-t-\alpha+1+w)+v(\alpha-w-1)-w.
\end{eqnarray*}

The inequality (\ref{MAIN}) is clearly strict if $\alpha \ge 3$.

 In order to complete the proof we have only
to check that this inequality is strict for $t=\alpha=2$ when
$v=1<t$. Suppose this is not the case.  From equality in (\ref{eq:2}) we have
$|ST|=|S|+|T|-2$. By Kneser Theorem, both $S$ and $T$ are
$H$--periodic by a subgroup $H$ of cardinality $|H|=\alpha=2$.

It follows by (\ref{KN2}) that
$$|N_2 (A',B)|\ge  |N_2
(S,T)|=|N_1 (S,T)|=|S|+|T|-2=|A|+|B|-4.$$ Moreover we may assume
$|N_1 (A,B)|\le |A|+|B|-1$ since otherwise $|N_1(A,B)|+|N_2
(A,B)|\ge 2(|A|+|B|-2)$ and we are done.

Hence $|N_1 (A',B)\setminus ST|=3$, and since $S$ and $T$ are
$H$--periodic, $N_1 (A',B)\setminus ST=\{1\}\cup aH,$ for some $a$.
We shall consider only the case $a\in S$, the other case being
essentially the same.

 Observe that $aH\cap A'T=\emptyset,$ otherwise $aH\subset N_2(A',B)$ because $aH=aH.1$.
 Also $1\notin A'T=A'TH.$ It follows that $|A'T|\le |A'B|-3,$ and hence $A'T=ST$.
 By Kneser's Theorem \ref{kneser},
 $$|ST|=|A'T|\ge |A'H|+|HT|-|H|=|A|+|H|+|T|-|H|=|S|+|T|-1,$$ a contradiction.
This completes the proof.\qed

\section*{Acknowledgement}
The authors would like to thank Professor W. Dicks for stimulating discussions
and comments on Pollard's inequality and also for several remarks
allowing us to remove some obscure parts in our first draft.

\end{document}